\documentclass[12pt]{amsart}
\usepackage{amssymb,latexsym,eucal}

\usepackage{hyperref}
\usepackage[letterpaper,centering,text={6.5in,9in}]{geometry}
\usepackage{graphicx}
\usepackage{breqn}
\usepackage{latexsym}
\usepackage[utf8]{inputenc}
\usepackage{amsmath}
\usepackage{amsthm}
\usepackage{amsfonts}
\usepackage{amssymb}
\usepackage{epsfig}
\usepackage{nicefrac}
\usepackage[all]{xy}
\usepackage{graphicx}
\usepackage{enumerate}
\usepackage{mathtools}    
\DeclarePairedDelimiterX\setc[2]{\{}{\}}{\,#1 \;\delimsize\vert\; #2\,}

\newtheorem{theorem}{Theorem}[section]

\newtheorem{proposition}[theorem]{Proposition}
\newtheorem{lemma}[theorem]{Lemma}

\theoremstyle{definition}

\newtheorem{remark}[theorem]{Remark}
\newtheorem{example}[theorem]{Example}
\newtheorem{definition}[theorem]{Definition}

\newcommand{\CC}{{\mathbb C}}

\newcommand{\RR}{{\mathbb R}}

\newcommand{\NN}{{\mathbb N}}

\begin{document}

\title[Higher capacities do not satisfy the symplectic Brunn-Minkowski inequality]{Higher index symplectic capacities do not satisfy the symplectic Brunn-Minkowski inequality}

\author{Ely Kerman}
\author{Yuanpu Liang}
\address{Department of Mathematics\\
University of Illinois at Urbana-Champaign\\
1409 West Green Street\\
Urbana, IL 61801, USA.}
\thanks{The first named author is supported by a grant from the Simons Foundation}

\date{\today}

\begin{abstract}
In \cite{ao},  Artstein-Avidan and Ostrover establish a symplectic version of the  classical Brunn-Minkowski inequality where the role of the volume is played by the Ekeland-Hofer-Zehnder capacity. Here we prove that this symplectic Brunn-Minkowski inequality fails to hold for all of  the higher index symplectic capacities defined by Gutt and Hutchings in \cite{gh}. 

\end{abstract}

\maketitle

\section{Introduction.}

The Brunn-Minkowski inequality asserts that if  $U_1$ and $U_2$  are two compact subsets of  $\RR^{n}$,  then 
\begin{equation*}
\label{bm}
{\mathrm{Volume}}(U_1+U_2)^{\frac{1}{n}} \geq {\mathrm{Volume}}(U_1)^{\frac{1}{n}}+{\mathrm{Volume}}(U_2)^{\frac{1}{n}},
\end{equation*}
where $U_1 + U_2 =\{x_1 + x_2 \mid x_1\in U_1,\, x_2 \in U_2\}$ is the Minkowski sum. 
In \cite{ao},  Artstein-Avidan and Ostrover prove a version of this fundamental geometric inequality in the context of symplectic geometry where the role of  the volume is played by a {\em symplectic capacity}, \cite{eh1}.  By now, many symplectic capacities have been defined with different tools and with different perspectives in mind, \cite{chls}. 
One common feature they share is that, in dimension two, they must be proportional to the volume (area), and hence must satisfy the Brunn-Minkowski inequality. However, in dimensions greater than two, every symplectic capacity  must take a finite positive value on the symplectic cylinder
\begin{equation*}
\label{cyl}
B^2(1) \times \RR^{2n-2} \subset (\RR^{2n}, \omega_{std})
\end{equation*}
and hence must be manifestly different from any measurement derived from the volume. In view of this latter fact, it is remarkable that, in \cite{ao}, the authors prove the following symplectic analogue of the Brunn-Minkowski inequality, 
\begin{equation}
\label{bmc}
c^{\mathrm{EHZ}}(K_1+K_2)^{\frac{1}{2}} \geq c^{\mathrm{EHZ}}(K_1)^{\frac{1}{2}}+c^{\mathrm{EHZ}}(K_2)^{\frac{1}{2}}.
\end{equation}
Here, $K_1$ and $K_2$ are {\em convex bodies} in $\RR^{2n}$ and $c^{\mathrm{EHZ}}$ is  the Ekeland-Hofer-Zehnder capacity.  

Given inequality \eqref{bmc}, and the many applications of it developed in \cite{ao}, it is natural to ask if this symplectic version of the Brunn-Minkowski inequality is another property shared by all symplectic capacities.
In this note, we settle this question in the negative.

\subsection{The Main Result}

The Ekeland-Hofer-Zehnder capacity is the first in an infinite sequence of symplectic capacities constructed by  Ekeland and Hofer in \cite{eh2}.  In the recent work \cite{gh},  Gutt and Hutchings use $S^1$-equivariant symplectic homology to construct another sequence of symplectic capacities that are conjectured to be equal to the Ekeland-Hofer capacities. Let $c_k$ denote the $k^{th}$ Gutt--Hutchings capacity.  Our main result is the following. 
 \begin{theorem}\label{one}  For all $k > 1$ and $n>1$ there are convex domains $K_1$ and $K_2$ in $\RR^{2n}$ such that  
 \begin{equation}
\label{nobmc}
c_k(K_1+K_2)^{\frac{1}{2}} < c_k(K_1)^{\frac{1}{2}}+c_k(K_2)^{\frac{1}{2}}.
\end{equation}
\end{theorem}

\medskip

To prove Theorem \ref{one} we first show that  it can be reduced to the four dimensional case. We then construct two families of examples that realize the stated inequality; one for even values of $k$ and the other for odd values of $k$. These examples involve the Minkowski sums of symplectic ellipsoids. The required computation of the capacities of these sums is made possible by two ingredients. The first is Gutt and Hutching's formula for the capacity of convex toric domains from \cite{gh}. The second ingredient is the elegant parameterization of  the boundary of the Minkowski sum of two ellipsoids  derived  by Chirikjian and Yan in \cite{cy}. In the next section we recall these formulae, in the relevant settings, and use them to compute the capacities of the sums of four dimensional symplectic ellipsoids. The proof of Theorem \ref{one} is contained the  third section of the paper. In the final section of the paper we describe an argument due to Ostrover which shows that, in each dimension, all but finitely many cases of Theorem 1 can be established using a result from \cite{ao}. This result (see Theorem \ref{mean} below) establishes an upper bound for the capacity of centrally symmetric convex bodies in terms of their mean-width, whenever the capacity satisfies the symplectic  Brunn-Minkowski inequality.

\begin{remark}
One property of the Ekeland-Hofer-Zehnder capacity that plays a crucial role in the proof of \eqref{bmc} is that, for a convex body $K$, the capacity $c^{\mathrm{ EHZ}}(K)$ can be defined as the {\em minimum} value of a functional. In particular,  $c^{\mathrm{EHZ}}(K)$ is the minimum action of a closed characteristic on the boundary of $K$. Morally speaking, the higher index Ekeland-Hofer capacities of $K$ correspond to the critical values of saddle points. It would be interesting to know if any capacity which satisfies inequality \eqref{bmc} must be proportional to the Ekeland-Hofer-Zehnder capacity on convex sets. 
\end{remark}

\subsection*{Acknowledgements}
The authors thank Yaron Ostrover for his helpful comments and for the argument which appears in the final section of the paper.

\section{Formulae}

\subsection{The Chirikjian--Yan parameterization} Let $B^n(1)$ be the open unit ball in $\RR^n$. For each ellipsoid $E$ in $\RR^n$ there is a unique $n \times n$ matrix $A$ such that $E=A(B^n(1))$.  For any two ellipsoids $E_1,\,E_2 \subset \RR^n$, the map $CY \colon S^{n-1} \to \partial(E_1 +E_2) $ defined by 
\begin{equation}
\label{cy}
CY(x) = A_1 x +A_2\left( \frac{A_2^TA_1^{-1}(x)}{\| A_2^TA_1^{-1}(x)\|} \right)
\end{equation}
is a diffeomorphism. We refer the reader to \cite{cy} for the geometric derivation of this formula.

\begin{example}\label{alligned} Let $E(a, b)$, $E(c, d)$ be two symplectic ellipsoids in $\RR^4 =\CC^2 $ where $$E(a,b) =\setc*{ (z_1,z_2) \in \CC^2}{\frac{ |z_1|^2}{a^2} + \frac{ |z_2|^2}{b^2}  <1}.$$
Label the points $z \in S^3$ using  polar coordinates in $\RR^4=\CC \times \CC$, so that 
$$
z=(\cos \psi, \theta_1, \sin \psi, \theta_2)
$$
for unique values of $\theta_1,\, \theta_2 \in [0, 2\pi)$ and $\psi \in [0, \pi/2].$ For symplectic ellipsoids, the map $CY$ preserves the two complex factors and fixes the angles $\theta_1$ and $\theta_2$. Formula \eqref{cy} implies that the points on $\partial(E(a,b) +E(c,d))$ are of the form
$$
(g(\psi), \theta_1, h(\psi), \theta_2)
$$
where 
$$g(\psi)=\cos \psi \left( a+ \frac{c^2}{a f(\psi)}\right),$$ $$h(\psi)=\sin \psi \left( b+ \frac{d^2}{b f(\psi)}\right),$$ and
\begin{equation}
\label{f}
f(\psi)=\sqrt{\frac{c^2}{a^2} \cos^2 \psi + \frac{d^2}{b^2} \sin^2 \psi}. 
\end{equation}
\end{example}

Let  $\mu \colon \CC^n \to \RR^n$ be the standard moment map $\mu(z_1, \dots z_n) =(\pi |z_1|^2,\dots, \pi |z_n|^2)$. A {\em toric domain} is a subset of $\RR^{2n}$ of the form $\mu^{-1}(\Omega)$ for some subset $\Omega$ of $\RR^n$. In Example \ref{alligned},  $\partial(E(a,b) +E(c,d))$ is the boundary of the toric domain $\mu^{-1}(\Omega)$ where $\Omega \subset \RR^2$ is the region bounded by the coordinate axes and the image of the curve  $\psi \mapsto(\pi (g(\psi))^2, \pi(h(\psi))^2)$, see Figure \ref{omega},.

\begin{figure}[!h]
\centering
\caption{The region $\Omega$ for which $E(a,b)+E(c,d)= \mu^{-1}(\Omega)$.}
\bigskip
  \includegraphics[width=9cm]{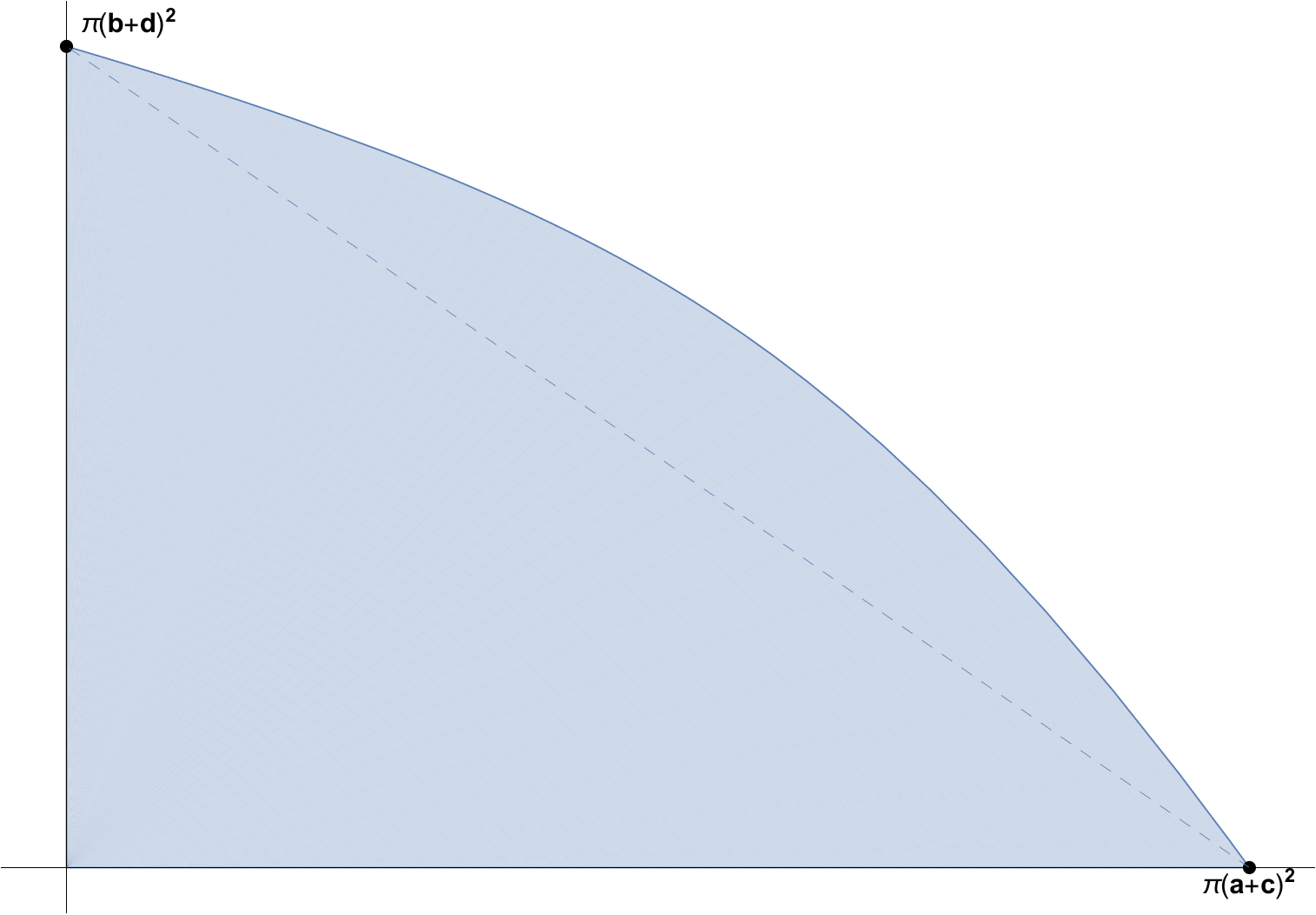}
  \label{omega}
\end{figure}

\begin{definition}[\cite{gh}, Definition 1.4]
A toric domain $\mu^{-1}(\Omega) \subset \RR^{2n}$ is {\em convex} if the region $$\{(x_1, \dots, x_n) \in \RR^n) \,\,\colon \,\,(|x_1|, \dots, |x_n|) \in \Omega\}$$
is compact and convex.
\end{definition}

\begin{lemma} The (closure of the) Minkowski sum $E(a,b) + E(c,d)$ is a convex toric domain in the sense of \cite{gh}.
\end{lemma}

\begin{proof} The Minkowski sum of a symplectic ellipsoid and a multiple of itself is just another symplectic ellipsoid. Hence the result is automatic when $\frac{c}{a} = \frac{d}{b}$. Relabelling if necessary, we may therefore assume that 
\begin{equation}
\label{cd}
\frac{c}{a} < \frac{d}{b}.
\end{equation} Let $C$ be the function defined by
\begin{equation}
\label{C}
C(\pi (g(\psi))^2) =\pi (h(\psi))^2. 
\end{equation}
It suffices to prove that, for all $\psi \in (0, \pi / 2)$,  the functions  $C'(\pi (g(\psi))^2)$ and $C''(\pi (g(\psi))^2)$ are both negative. 

We first observe that, by \eqref{cd}, the function 
\begin{equation*}
\label{ }
g'(\psi) = -\sin \psi \left( a+ \frac{c^2}{a f(\psi)} \right) +\cos^3 \psi \sin \psi \frac{c^2}{af^3} \left(\frac{c^2}{a^2} - \frac{d^2}{b^2}\right)
\end{equation*}
is negative on $(0, \pi / 2)$.  Equation \eqref{C} then implies that 
\begin{equation*}
\label{ }
C'(\pi g^2) = \frac{h h'}{g g'}
\end{equation*}
on $(0, \pi / 2)$, and a straightforward computation yields
\begin{equation}
\label{h/g}
\frac{h h'}{g g'} = -\frac{b^2 f +d^2}{a^2f +c^2}.
\end{equation}
Hence, $C'$ is negative in the desired domain. 

As for the second derivative, equations \eqref{C} and \eqref{h/g} imply that 
\begin{eqnarray*}
C''(\pi g^2) & = & \frac{-1}{2\pi g g'}\left(\frac{b^2 f +d^2}{a^2f +c^2}\right)' \\
{} & = &  \frac{1}{2\pi g g'} \frac{a^2 b^2}{(a^2f +c^2)^2} \left(\frac{c^2}{a^2} -\frac{d^2}{b^2}\right)^2 \cos\psi \sin \psi.
\end{eqnarray*}
This is negative on the desired range since $g'$ is.
\end{proof}

\subsection{The Gutt-Hutchings formula} In \cite{gh}, the authors prove that  the capacities of any convex toric domain  $\mu^{-1}(\Omega)$ in  $\RR^{2n}$ are given by
\begin{equation*}
\label{ }
c_k (\mu^{-1}(\Omega)) = \min \setc*{ \|v\|^*_{\Omega}}{v=(v_1, \dots, v_n),\, v_j \in \{0\} \cup \mathbb{N}, \, \sum_1^n v_j =k}
\end{equation*}
where
\begin{equation*}
\label{ }
\|v\|^*_{\Omega} =\max\setc*{ \langle v,w \rangle}{w \in \Omega}.
\end{equation*}

\begin{example}[\cite{gh}, Example 1.9] The closure of the ellipsoid $E(a,b)$ is equal to $\mu^{-1}(\Omega)$ for 
$$
\Omega =\setc*{(x_1,x_2) \in \RR^2_{\geq 0}}{\frac{x_1}{\pi a^2} + \frac{x_2}{\pi b^2} \leq1}.
$$
In this case 
\begin{equation*}
\label{ }
\|v\|^*_{\Omega} =\max\{ \pi a^2 v_1,\, \pi b^2 v_2\}.
\end{equation*}
and 
\begin{equation*}
\label{ellipse}
c_k (E(a,b)) = \left(\mathrm{Sort}\left\{ \NN \pi a^2 \cup \NN \pi b^2 \right\}\right)[k],
\end{equation*}
the $k^{th}$ element in the sequence of positive integter multiples of $\pi a^2$ and $\pi b^2$ ordered by size with repetitions.
\end{example}

\subsection{Computing $c_k(E(a,b) + E(c,d))$}
For  $\Omega = \mu(E(a,b) + E(c,d))$ as in Example \ref{alligned}, we have 
\begin{equation}
\label{max}
\|v\|^*_{\Omega} =\max_{\psi \in [0, \pi/2]}\{\pi v_1(g(\psi))^2+\pi v_2(h(\psi))^2\}
 \end{equation}

The next result can be used to transform this, and hence the formula for  $c_k(E(a,b)+ E(c,b))$, into a direct computation. 

\begin{lemma}\label{direct} Let $\Omega = \mu(E(a,b)+ E(c,d))$ and assume that $\frac{c}{a}< \frac{d}{b}$. If  
 $\displaystyle{\frac{v_1c^2-v_2d^2}{v_2b^2-v_1a^2}}$ is in  the interval $\left( \frac{c}{a},\frac{d}{b}\right)$ and $v_2b^2-v_1a^2<0$, then 
 \begin{equation*}
\label{ }
\|v\|^*_{\Omega} = \pi \left( b^2c^2 - a^2d^2 \right)v_1v_2\left( \frac{1}{v_1c^2-v_2d^2} +  \frac{1}{v_2b^2-v_1a^2}\right).
\end{equation*}
Otherwise, 
\begin{equation*}
\label{ }
\|v\|^*_{\Omega} =\max\left\{ v_1 \pi (a+c)^2,\, v_2 \pi (b+d)^2\right\}.
\end{equation*}

\end{lemma}

\begin{proof}
Let $\Delta = \frac{d}{b} -\frac{c}{a}.$ By \eqref{f} we have $$\sin^2\psi =  \frac{1}{\Delta}\left(f^2 -\frac{c^2}{a^2}\right) \text{ and } \cos^2\psi =  \frac{1}{\Delta}\left(\frac{d^2}{b^2}-f^2\right).$$
Since the function $f(\psi)$ is monotonically increasing on $[0, \pi/2]$ and goes from $\frac{c}{a}$ to $\frac{d}{b}$, equation \eqref{max} can then be rewritten as 
\begin{equation*}
\label{}
\|v\|^*_{\Omega} =\max_{f \in [\frac{c}{a},\frac{d}{b}]}\left\{\underbrace{\pi v_1a^2\frac{1}{\Delta}\left(\frac{d^2}{b^2}-f^2\right)\left(1+ \frac{c^2}{a^2f}\right)^2+\pi v_2b^2\frac{1}{\Delta}\left(f^2+\frac{c^2}{a^2}\right)\left(1+ \frac{d^2}{b^2f}\right)^2}_{S(f)}\right\}.
 \end{equation*}
 A straightforward computation yields 
 \begin{equation*}
\label{ }
S'(f) = \frac{2\pi}{\Delta f^3} \frac{1}{v_2b^2-v_1a^2}\left[f^3 \left(f-\frac{v_1c^2-v_2d^2}{v_2b^2-v_1a^2}\right) + \frac{c^2d^2}{a^2b^2}\left(f-\frac{v_1c^2-v_2d^2}{v_2b^2-v_1a^2}\right) \right].
\end{equation*}
So, if $f_0 = \frac{v_1c^2-v_2d^2}{v_2b^2-v_1a^2}$ is contained in $\left(\frac{c}{a},\frac{d}{b}\right)$, then it is the unique critical point of $S$ in this interval. Moreover, since
 \begin{equation*}
\label{ }
S''(f_0) = \frac{2\pi}{\Delta f_0^3} \frac{1}{v_2b^2-v_1a^2}\left[f_0^3  + \frac{c^2d^2}{a^2b^2}\right], 
\end{equation*}
the critical point $f_0$ is a global maximum if $v_2b^2-v_1a^2 <0$ and we have
 \begin{equation*}
\label{ }
\|v\|^*_{\Omega} =S(f_0) = \pi \left( b^2c^2 - a^2d^2 \right)v_1v_2\left( \frac{1}{v_1c^2-v_2d^2} +  \frac{1}{v_2b^2-v_1a^2}\right). 
\end{equation*} 
On the other hand, if $f_0 = \frac{v_1c^2-v_2d^2}{v_2b^2-v_1a^2}$ is not contained in $\left(\frac{c}{a},\frac{d}{b}\right)$ or if $v_2b^2-v_1a^2 \geq 0$, then $$\|v\|^*_{\Omega} =\max_{f \in [\frac{c}{a},\frac{d}{b}]}\left\{S(f)\right\} = \max \left\{S\left(\frac{c}{a}\right),S\left(\frac{d}{b}\right)\right\}=\max\left\{ v_1 \pi (a+c)^2,\, v_2 \pi (b+d)^2\right\}.$$

\end{proof}

\begin{remark}
The Minkowski sum  $E(a, b)+ E(c,d)$ naturally contains the symplectic ellipsoid $E(a+c, b+d)$ and the boundaries of these domains intersect exactly along the $2$ standard embedded closed Reeb orbits on $\partial E(a+c, b+d)$. (See Figure \ref{omega}.) The monotonicity property of the capacities $c_k$ implies that  $$c_k(E(a,b)+ E(c,d)) \geq c_k(E(a+c, b+d)),$$  for all $k \in \mathbb{N}$. With this in mind one can use Lemma \ref{direct} to determine when this inequality is strict. One simply applies the criteria of the lemma to the vector $v=(v_1,k-v_1)$ for which 
$$c_k(E(a+c, b+d))=\|v\|^*_{\mu(E(a+c, b+d))}.$$  If  
 $\displaystyle{\frac{v_1c^2-(k-v_1)d^2}{(k-v_1)b^2-v_1a^2}}$ is in  $\left( \frac{c}{a}, \frac{d}{b}\right)$ and $(k-v_1)b^2<v_1a^2$, then $c_k(E(a,b)+ E(c,d))> c_k(E(a+c, b+d)).$

\end{remark}

\begin{remark}
Let $\mathrm{GW}(K)$ denote the Gromov width of $K$. By monotoncity, we have $$\mathrm{GW}(E(a,b)+ E(c,d)) \geq \mathrm{GW}(E(a+c, b+d)).$$ Since $\mathrm{GW}(E(a,b)) =c_1(E(a,b)) = c^{\mathrm{EHZ}}(E(a,b)) $, it follows from this and \eqref{bmc} (or indeed Lemma \ref{direct}) that 
$$\mathrm{GW}(E(a,b)+ E(c,d))^{\frac{1}{2}} \geq \mathrm{GW}(E(a,b))^{\frac{1}{2}} + \mathrm{GW}(E(c,d))^{\frac{1}{2}}.$$  In other words, the Gromov  width satisfies the symplectic Brunn-Minkowski inequaity on ellipsoids. As pointed out to  the authors by Ostrover, it is not known whether this holds for general convex bodies. 
\end{remark}

\section{Proof of Theorem \ref{one}.}
First we show that it suffices to prove the following result in four dimensions.
\begin{proposition}\label{4d} For  each $k \geq 2$ there are positive numbers $a,\,b,\, c$ and  $d$ such that  
\begin{equation}
\label{failure}
c_k(E(a, b)+E(c, d))^{\frac{1}{2}} < c_k(E(a, b))^{\frac{1}{2}}+ c_k(E(c, d))^{\frac{1}{2}}.
\end{equation}
\end{proposition}

\bigskip

\subsection{Proof that Proposition \ref{4d} implies Theorem \ref{one}} For convex toric domains, $X \subset \RR^{2n}$ and $Y \subset \RR^{2m}$, the Gutt--Hutchings capacities have the following product property (see \cite{gh}, Remark 1.10): 
$$
c_k(X \times Y) = \min_{i+j=k}\{ c_i(X) + c_j(Y)\}.
$$
Hence for any convex toric domain $X \subset \RR^{2n}$ and any $k \in \NN$  we have 
$$
c_k(X \times B^m(R)) = c_k(X),
$$
for all $m \in \NN$ and  $R >\sqrt{c_k(X)/\pi}.$ 

Fix $k>1$ and $n>2$. Choose  $a,\,b,\, c$ and  $d$ as in Proposition \ref{4d} and choose $$R > \sqrt{c_k(E(a, b)+E(c, d))/\pi}.$$
For $K_1= E(a,b) \times B^{2n-4}(R)$ and $K_2= E(c,d) \times B^{2n-4}(R)$ we then have 
\begin{eqnarray*}
c_k(K_1+K_2) & = & c_k((E(a,b) \times B^{2n-4}(R))+(E(c,d) \times B^{2n-4}(R))) \\
{} & = & c_k((E(a,b)+E(c,d)) \times B^{2n-4}(2R))  \\
{} & = & c_k(E(a,b)+E(c,d)).
\end{eqnarray*}
Since $c_k(E(a,b) \times B^{2n-4}(R)) = c_k(E(a,b))$ and $c_k(E(c,d) \times B^{2n-4}(R)) = c_k(E(c,d))$ it follows that Proposition \ref{4d} implies Theorem \ref{one}.

\subsection{Proof of Proposition \ref{4d}} 

We construct two families of examples for which inequality \eqref{failure} holds. One family for even values of $k$ and another for odd values of $k$. 

\begin{proposition}\label{even} For every even $k $, we have
\begin{equation*}
\label{}
    c_k\left(E\left(1+1/k, 1\right)+E\left(1, 1+1/k\right)\right)^{\frac{1}{2}} < c_k\left(E\left(1+1/k\right),1\right)^{\frac{1}{2}}+c_k\left(E\left(1, 1+1/k\right)\right)^{\frac{1}{2}}.
\end{equation*}
\end{proposition}

\bigskip

\begin{proof}
By \eqref{ellipse}, we have 
\begin{equation*}
 c_k(E(1+1/k,1)) = c_k(E(1,1+1/k)) =\pi\left(k/2+1\right),
\end{equation*}
and hence 
\begin{equation*}
\left(c_{k}(E(1,1+1/k))^{\frac{1}{2}} + c_{k}(E(1+1/k,1))^{\frac{1}{2}}\right)^{2} =  \pi \left( 2k+4\right).
\end{equation*}
With this, it suffices to prove the following.
\begin{equation}
\label{ceven}
c_k(E(1,1+1/k)+E(1+1/k,1)) = \pi \left(2k + 2 + \frac{1}{k}\right).
\end{equation}
We first apply Lemma \ref{direct} to $E\left(1+1/k, 1\right)+E\left(1, 1+1/k\right)$ and  the vector $v=(k/2,k/2)$.  In this case 
$$\displaystyle{\frac{v_1c^2-v_2d^2}{v_2b^2-v_1a^2}}=1 \in \left(\frac{1}{1+1/k}, \,1+1/k \right) = \left(\frac{c}{a},\,\frac{d}{b} \right),$$ 
and $v_2b^2-v_1a^2 =\frac{k}{2}(-\frac{2}{k} -\frac{1}{k^2}) <0$. Hence, Lemma \ref{direct} implies that  
\begin{equation*}
\label{ }
\|(k/2,k/2)\|^*_{\Omega} = \pi \left(2k + 2 + \frac{1}{k}\right).
\end{equation*}
To verify \eqref{ceven}, it remains to show that for any other $v=(v_1, v_2)$ with $v_1+v_2=k$ we have $$\|v\|^*_{\Omega} \geq \pi \left(2k + 2 + \frac{1}{k}\right).$$
For $v \neq (k/2,k/2)$, either $v_1>k/2$ or $v_2>k/2$. Together with \eqref{max} this yields
\begin{eqnarray*}
\|v\|^*_{\Omega} & = & \max_{\psi \in [0, \pi/2]}\{\pi v_1(g(\psi))^2+\pi v_2(h(\psi))^2\} \\
{} & \geq & \pi \left(\frac{k}{2}+1\right)\left(2+\frac{1}{k}\right)^2  \\
{} & > & \pi \left( 2k+6 \right)\\
 {} & > & \pi \left(2k + 2 + \frac{1}{k}\right),
\end{eqnarray*}
as required.
\end{proof}

\begin{proposition}\label{odd} For every odd $k>1$, we have
\begin{equation*}
\label{}
    c_k\left(E(1, 1)+E(1-1/k, 1)\right)^{\frac{1}{2}} < c_k\left(E\left(1, 1\right)\right)^{\frac{1}{2}}+c_k\left(E\left(1-1/k, 1\right)\right)^{\frac{1}{2}}.
\end{equation*}
\end{proposition}

\begin{proof} First we show that 
\begin{equation}
\label{codd}
 c_k\left(E(1, 1)+E(1-1/k, 1)\right) = \pi ((k+1)/2)(1-1/k)^2.
\end{equation}
Let $k=2n+1$ for $n>0$. Applying Lemma \ref{direct} to $E(1, 1)+E(1-1/k, 1)$ and the vector $v=(n+1,n)$, we have  $$\displaystyle{\frac{v_1c^2-v_2d^2}{v_2b^2-v_1a^2}}<0,$$ and hence 
\begin{equation*}
\label{ }
\|(n+1,n)\|^*_{\Omega} = \max\{\pi(n+1)(2-1/k)^2, \pi 4 n\}= \pi(n+1)(2-1/k)^2.
\end{equation*}
It remains to show that for any other $v=(v_1, v_2)$ with $v_1+v_2=k$ we have $$\|v\|^*_{\Omega} \geq \pi(n+1)(2-1/k)^2.$$
For $v \neq (n+1, n)$, either $v_1 \geq n+2$ or $v_2 \geq n+1$.
In the first case, we have
\begin{eqnarray*}
\|v\|^*_{\Omega} & = & \max_{\psi \in [0, \pi/2]}\{\pi v_1(g(\psi))^2+\pi v_2(h(\psi))^2\} \\
{} & \geq & \pi (n+2)( g(0))^2  \\
{} &= & \pi (n+2)(2-1/k)^2,   
\end{eqnarray*}
as desired.
In the second case, we have
\begin{eqnarray*}
\|v\|^*_{\Omega} & = & \max_{\psi \in [0, \pi/2]}\{\pi v_1(g(\psi))^2+\pi v_2(h(\psi))^2\} \\
{} & \geq & \pi (n+1)( h(\pi/2))^2  \\
{} &= & \pi (4n+4). 
\end{eqnarray*}
Since $(n+1)(2-1/k)^2 =4n+4- \frac{k+1}{k} + \frac{k+1}{2k^2}$, we are again done. This completes the proof of \eqref{codd}.

By \eqref{ellipse},  we have $c_k(E(1,1))=\pi(n+1)$ and $c_k(E(1-1/k,1)) = \pi n$. It remains for us to show that 
$$
\sqrt{n}+ \sqrt{n + 1}  > \sqrt{n + 1}\left( 2- \frac{1}{2n+1}\right).
$$
But this is equivalent to 
$$
\frac{\sqrt{n}} {\sqrt{n + 1} }  > \frac{2n}{2n+1},
$$
which is trivial to verify for all $n>0$.

\end{proof}

\section{An observation of Ostrover}

The following result is established in \cite{ao}.

\begin{theorem}[\cite{ao}, Corollary 1.7] \label{mean}  Let  $C$ be a symplectic capacity on $(\RR^{2n}, \omega_{std})$ which is normalized by the condition $C(B^{2n}(1)) =\pi$ and satisfies the symplectic Brunn-Minkowski inequality. Then, for every centrally symmetric convex body $K \subset \RR^{2n}$ one has 
\begin{equation*}
\label{ }
C(K) \leq \pi (M(K))^2,
\end{equation*}
where $M(K)$ is the mean-width of $K$.
\end{theorem}

Here, the mean-width of $K$ is defined by $$M(K) = \int_{S^{2n-1}} \max_{y \in K}\langle x,y\rangle \sigma(dx),$$  where $\sigma$ is the rotationally invariant probability measure on the unit sphere $S^{2n-1} \subset \RR^{2n}$.

It was observed by Ostrover that Theorem \ref{mean} can be used, in the following manner,  to prove Theorem \ref{one} except when $k$ is equal to $3$, $5$, or $7$.  Arguing as in Section 3.1, we can restrict ourselves to the case $n=2$. Consider the symplectic polydisk $$P(1,1) =\left\{ (z_1,z_2) \in \CC^2  \,\,\colon\,\, |z_1|<1, |z_2| <1\right\}.$$
For all $k \in \NN$ we have 
\begin{equation*}
\label{ }
c_k(P(1,1)) = \pi k,
\end{equation*}
and a straightforward computation yields 
\begin{equation*}
\label{ }
M(P(1,1)) =\frac{4}{3}.
\end{equation*}

Set $\bar{c}_k = \frac{\pi}{ \lfloor \frac{k+1}{2} \rfloor } c_k$. It follows from  Theorem \ref{mean} and the formulas above, that $\bar{c}_k$
(and hence $c_k$) fails to satisfy the symplectic Brunn-Minkowski inequality whenever
\begin{equation*}
\label{ }
\frac{k}{ \lfloor \frac{k+1}{2} \rfloor } > \frac{16}{9}.
\end{equation*}
This holds for all even values of $k$ and all odd values of $k$ greater than $7$.

\end{document}